\documentclass[11pt,a4paper]{article}
\usepackage{epsf,epsfig,amsfonts,amsgen,amsmath,amstext,amsbsy,amsopn,amsthm}
\usepackage{ebezier,eepic}
\usepackage{color}
\newtheorem{theorem}{Theorem}
\newtheorem{proposition}{Proposition}
\newtheorem{lemma}{Lemma}
\newtheorem{false statement}{False statement}

\theoremstyle{definition}

\newtheorem{claim}{Claim}
\newtheorem{subclaim}{Claim}[claim]

\newtheorem{conjecture}{Conjecture}
\newtheorem{problem}{Problem}

\newcommand{\cl}{\rm cl}
\newcounter{mathitem}
\newenvironment{mathitem}
  {\begin{list}{{$(\roman{mathitem})$}}{
   \setcounter{mathitem}{0}
   \usecounter{mathitem}
   \setlength{\topsep}{0pt plus 2pt minus 0pt}
   \setlength{\parskip}{0pt plus 2pt minus 0pt}
   \setlength{\partopsep}{0pt plus 2pt minus 0pt}
   \setlength{\parsep}{0pt plus 2pt minus 0pt}
   \setlength{\leftmargin}{20pt}
   \setlength{\itemsep}{0pt plus 2pt minus 0pt}}}
  {\end{list}}

\begin{document}

\title{\bf\Large Induced subgraphs with large degrees at end-vertices for hamiltonicity of claw-free graphs\thanks{Supported by NSFC (11271300), the Doctorate Foundation of Northwestern Polytechnical University (cx201202 and cx201326) and the project NEXLIZ --
CZ.1.07/2.3.00/30.0038, which is co-financed by the European Social
Fund and the state budget of the Czech Republic. }}

\date{}

\author{Roman \v{C}ada$^a$, Binlong Li$^{a,b}$\thanks{E-mail address: {\tt libinlong@nwpu.edu.cn (B. Li)}}, Bo Ning$^b$\thanks{E-mail address: {\tt bo.ning@tju.edu.cn (B. Ning)}} and
Shenggui Zhang$^{b}$\thanks{E-mail address: {\tt sgzhang@nwpu.edu.cn (S. Zhang).}}\\[2mm]
\small $^a$ Department of Mathematics, NTIS - New Technologies for the Information Society, \\
\small University of West Bohemia, 30614 Pilsen, Czech Republic\\
\small $^b$ Department of Applied Mathematics, School of Science, \\
\small Northwestern Polytechnical University, Xi'an, Shaanxi 710072,
P.R.~China}

\date{}
\maketitle

\begin{abstract}
A graph is called \emph{claw-free} if it contains no
induced subgraph isomorphic to $K_{1,3}$. Matthews and Sumner proved
that a 2-connected claw-free graph $G$ is hamiltonian if every
vertex of it has degree at least $(|V(G)|-2)/3$. At the workshop
C\&C (Novy Smokovec, 1993), Broersma conjectured the degree
condition of this result can be restricted only to end-vertices of
induced copies of $N$ (the graph obtained from a triangle by adding
three disjoint pendant edges). Fujisawa and Yamashita showed that
the degree condition of Matthews and Sumner can be restricted only
to end-vertices of induced copies of $Z_1$ (the graph obtained from
a triangle by adding one pendant edge). Our main result in this
paper is a characterization of all graphs $H$ such that a
2-connected claw-free graph $G$ is hamiltonian if each end-vertex of
every induced copy of $H$ in $G$ has degree at least $|V(G)|/3+1$.
This gives an affirmative solution
of the conjecture of Broersma up to an additive constant.

\medskip
\noindent {\bf Keywords:} induced subgraph; large degree;
end-vertex; claw-free graph; hamiltonian graph

\smallskip
\end{abstract}   

\section{Introduction}

We use Bondy and Murty  \cite{Bondy_Murty} for terminology and
notation not defined here and consider finite simple graphs only.

Let $G$ be a graph. For a vertex $v\in V(G)$ and a subgraph $H$ of
$G$, we use $N_H(v)$ to denote the set, and $d_H(v)$ the number, of
neighbors of $v$ in $H$, respectively. We call $d_H(v)$ the
\emph{degree} of $v$ in $H$. For $x,y\in V(G)$, an
$(x,y)$-\emph{path} is a path connecting $x$ and $y$. If $x,y\in
V(H)$, the \emph{distance} between $x$ and $y$ in $H$, denoted
$d_H(x,y)$, is the length of a shortest $(x,y)$-path in $H$. When no
confusion occurs, we will denote $N_G(v)$, $d_G(v)$ and $d_G(x,y)$
by $N(v)$, $d(v)$ and $d(x,y)$, respectively.

Let $G$ be a graph and $G'$ a subgraph of $G$. If $G'$ contains all
edges $xy\in E(G)$ with $x,y\in V(G')$, then $G'$ is called an
\emph{induced subgraph} of $G$ (or a subgraph \emph{induced by}
$V(G')$). For a given graph $H$, we say that $G$ is $H$-\emph{free}
if $G$ contains no induced copy of $H$. If $G$ is $H$-free, then we
call $H$ a \emph{forbidden subgraph} of $G$. Note that if $H_1$ is
an induced subgraph of a graph $H_2$, then an $H_1$-free graph is
also $H_2$-free.

We first give a fundamental sufficient degree condition for hamiltonicity
of graphs.

\begin{theorem}[Dirac \cite{Dirac}]\label{ThDi}
Let $G$ be a graph on $n\geq 3$ vertices. If every vertex of $G$ has
degree at least $n/2$, then $G$ is hamiltonian.
\end{theorem}

The graph $K_{1,3}$ is called the \emph{claw}, and its only vertex
of degree 3 is called its \emph{center}. For a given graph $H$, we
call a vertex $v$ of $H$ an \emph{end-vertex} of $H$ if $d_H(v)=1$.
Thus a claw has three end-vertices. In this paper, we use the common term
claw-free graphs for $K_{1,3}$-free graphs.

Hamiltonian properties of claw-free graphs have been well studied by
many graph theorists. The lower bound on the degrees in Dirac's
theorem can be lowered to roughly $n/3$ in the case of (2-connected)
claw-free graphs.

\begin{theorem}[Matthews and Sumner \cite{MatthewsSumner}]\label{ThMaSu}
Let $G$ be a 2-connected claw-free graph on $n$ vertices. If every
vertex of $G$ has degree at least $(n-2)/3$, then $G$ is
hamiltonian.
\end{theorem}

Forbidden subgraph conditions for hamiltonicity of graphs have also
received much attention. As $K_2$-free graphs are precisely the
edgeless graphs, it is natural to assume that, throughout this
paper, all forbidden subgraphs under consideration will have at
least three vertices. We also note that every connected $P_3$-free
graph is a complete graph, and thus it is trivially hamiltonian if it
has at least 3 vertices. It is in fact easy to show that $P_3$ is
the only connected graph $R$ such that every 2-connected $R$-free
graph is hamiltonian.

Bedrossian \cite{Bedrossian} characterized all the pairs of
forbidden subgraphs for hamiltonicity, excluding $P_3$.

\begin{theorem}[Bedrossian \cite{Bedrossian}]\label{ThBe}
Let $R$ and $S$ be connected graphs with $R,S\neq P_{3}$ and let $G$
be a 2-connected graph. Then $G$ being $R$-free and $S$-free implies
$G$ is hamiltonian if and only if (up to symmetry) $R=K_{1,3}$ and
$S=P_4,P_5,P_6,C_3,Z_1,Z_2,B,N$ or $W$ (see Fig.~1).
\end{theorem}

\begin{center}
\begin{picture}(360,200)

\thicklines

\put(5,140){\multiput(20,30)(50,0){5}{\put(0,0){\circle*{6}}}
\put(20,30){\line(1,0){100}} \put(170,30){\line(1,0){50}}
\qbezier[4](120,30)(145,30)(170,30) \put(20,35){$v_1$}
\put(70,35){$v_2$} \put(120,35){$v_3$} \put(170,35){$v_{i-1}$}
\put(220,35){$v_i$} \put(115,10){$P_i$}}

\put(265,130){\put(20,30){\circle*{6}} \put(70,30){\circle*{6}}
\put(45,55){\circle*{6}} \put(20,30){\line(1,0){50}}
\put(20,30){\line(1,1){25}} \put(70,30){\line(-1,1){25}}
\put(40,10){$C_3$}}

\put(0,0){\put(20,30){\circle*{6}} \put(70,30){\circle*{6}}
\multiput(45,55)(0,25){4}{\put(0,0){\circle*{6}}}
\put(20,30){\line(1,0){50}} \put(20,30){\line(1,1){25}}
\put(70,30){\line(-1,1){25}} \put(45,55){\line(0,1){25}}
\put(45,105){\line(0,1){25}} \qbezier[4](45,80)(45,92.5)(45,105)
\put(50,80){$v_1$} \put(50,105){$v_{i-1}$} \put(50,130){$v_i$}
\put(40,10){$Z_i$}}

\put(90,0){\put(45,40){\circle*{6}} \put(45,40){\line(-1,1){25}}
\put(45,40){\line(1,1){25}} \put(20,65){\line(1,0){50}}
\multiput(20,65)(50,0){2}{\multiput(0,0)(0,30){2}{\put(0,0){\circle*{6}}}
\put(0,0){\line(0,1){30}}} \put(25,10){$B$ (Bull)}}

\put(180,0){\multiput(20,30)(50,0){2}{\multiput(0,0)(0,30){2}{\put(0,0){\circle*{6}}}
\put(0,0){\line(0,1){30}}}
\multiput(45,85)(0,30){2}{\put(0,0){\circle*{6}}}
\put(45,85){\line(0,1){30}} \put(20,60){\line(1,0){50}}
\put(20,60){\line(1,1){25}} \put(70,60){\line(-1,1){25}}
\put(25,10){$N$ (Net)}}

\put(270,0){\put(45,30){\circle*{6}} \put(20,55){\line(1,0){50}}
\put(45,30){\line(1,1){25}} \put(45,30){\line(-1,1){25}}
\multiput(20,55)(0,30){2}{\put(0,0){\circle*{6}}}
\multiput(70,55)(0,30){3}{\put(0,0){\circle*{6}}}
\put(20,55){\line(0,1){30}} \put(70,55){\line(0,1){60}}
\put(10,10){$W$ (Wounded)}}

\end{picture}

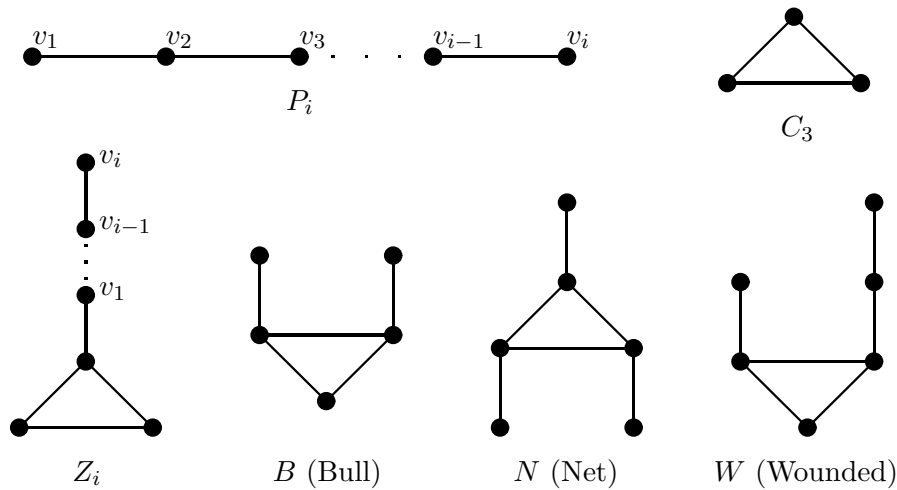
\small Fig. 1. Graphs $P_i,C_3,Z_i,B,N$ and $W$.
\end{center}

Note here that the claw is always one of the forbidden subgraphs. Also recall that a $P_4$-free graph is $P_5$-free, etc., so
the relevant graphs for $S$ (in Theorem \ref{ThBe}) are in fact
$P_6$, $N$ and $W$. All the other listed graphs are induced
subgraphs of $P_6$, $N$ or $W$.

At the workshop Cycles and Colourings 93 (Slovakia), Broersma
\cite{Broersma} proposed the following conjecture.

\begin{conjecture}[Broersma \cite{Broersma}]\label{CoBr}
Let $G$ be a 2-connected claw-free graph on $n$ vertices. If every
vertex of $G$ which is an end-vertex of an induced copy of $N$ in
$G$, has degree at least $(n-2)/3$, then $G$ is hamiltonian.
\end{conjecture}

This conjecture is still open. Fujisawa and Yamashita
\cite{FujisawaYamashita} obtained a similar result as follows.

\begin{theorem}[Fujisawa and Yamashita \cite{FujisawaYamashita}]\label{ThFuYa}
Let $G$ be a 2-connected claw-free graph on $n$ vertices. If every
vertex which is an end-vertex of an induced copy of $Z_1$ in $G$ has
degree at least $(n-2)/3$, then $G$ is hamiltonian.
\end{theorem}

Let $G$ be a graph on $n$ vertices and $H$ a given graph. We say
that $G$ satisfies $\varPhi(H,k)$ if for every vertex $v$ which is
an end-vertex of an induced copy of $H$ in $G$, $d(v)\geq(n+k)/3$.

In any connected graph, a vertex which is not an end-vertex of an
induced $P_3$ will be adjacent to all other vertices. Thus a graph
satisfying $\varPhi(P_3,-2)$ implies that every vertex of it has
degree at least $(n-2)/3$. By Theorem \ref{ThMaSu}, such a graph is
hamiltonian if it is 2-connected and claw-free. Also note that
Theorem \ref{ThFuYa} implies that every 2-connected claw-free graph
satisfying $\varPhi(Z_1,-2)$ is hamiltonian. Motivated by Conjecture
\ref{CoBr} and Theorem \ref{ThFuYa}, we consider in this paper, the following question:
For which graphs $H$, every 2-connected claw-free graph satisfying
$\varPhi(H,-2)$ is hamiltonian?

First, for a given connected graph $H$, note that if a graph is
$H$-free, then it naturally satisfies $\varPhi(H,-2)$. To guarantee
a 2-connected claw-free graph satisfying $\varPhi(H,-2)$ is
hamiltonian, by Theorem \ref{ThBe}, we can get that $H$ must be one
of the graphs in $\{P_3,P_4,P_5,P_6,\break C_3,Z_1,Z_2,B,N,W\}$ (to
avoid the discussion of trivial cases, we assume that $H$ has at
least three vertices). Note that $C_3$ has no end-vertex, and every
graph satisfies $\varPhi(C_3,-2)$ naturally. Since not every
2-connected claw-free graph is hamiltonian, $C_3$ does not meet our
result. Another counterexample is $Z_2$. The graph in Fig. 2 is
2-connected claw-free and satisfies $\varPhi(Z_2,-2)$ but it is not
hamiltonian. Thus we have the following result.

\begin{center}
\begin{picture}(140,120)
\thicklines
\multiput(30,60)(80,0){2}{\thinlines\put(0,0){\oval(20,100)}
\put(-5,15){$K_k$}}
\multiput(30,20)(0,40){3}{\multiput(0,0)(40,0){3}{\circle*{4}}
\put(0,0){\line(1,0){80}}}
\end{picture}

\small Fig. 2. A graph satisfying $\varPhi(Z_2,-2)$.
\end{center}

\begin{proposition}\label{PrFi}
Let $H$ be a connected graph on at least 3 vertices and let $G$ be a
2-connected claw-free graph. If $G$ satisfying $\varPhi(H,-2)$
implies $G$ is hamiltonian, then $H=P_3,P_4,P_5,P_6,Z_1,$ $B,N$ or
$W$.
\end{proposition}

What about the converse? Is every 2-connected claw-free graph
satisfying $\varPhi(H,-2)$ hamiltonian for all the graphs $H$ listed
in Proposition \ref{PrFi}?

Note that if a graph $G$ satisfies $\varPhi(P_i,k)$,
then it also satisfies $\varPhi(P_j,k)$ for $j\geq i$. Also note
that if $G$ satisfies $\varPhi(Z_1,k)$, then it also satisfies
$\varPhi(B,k)$; and if $G$ satisfies $\varPhi(B,k)$, then it also
satisfies $\varPhi(N,k)$. (We remark that a graph satisfying
$\varPhi(Z_2,k)$ cannot ensure it satisfies $\varPhi(W,k)$, although
$Z_2$ is an induced subgraph of $W$.) So, in the following, we just
consider the three graphs $P_6$, $N$ and $W$. We propose the
following problem:

\begin{problem}\label{PrIf}
Let $H=P_6$, $N$ or $W$. Is every 2-connected claw-free graph
satisfying $\varPhi(H,-2)$ hamiltonian?
\end{problem}

We believe that the answer to Problem \ref{PrIf} is positive, but
the proof may need more technical discussions. However, we can prove
a slightly weaker result as follows.

\begin{theorem}\label{ThIf}
Let $H=P_6$, $N$ or $W$, and let $G$ be a 2-connected claw-free
graph. If $G$ satisfies $\varPhi(H,3)$, then $G$ is hamiltonian.
\end{theorem}

Note that the graph in Fig. 2 satisfies $\varPhi(Z_2,3)$ when $k\geq
6$. Combining with Proposition \ref{PrFi} and Theorem \ref{ThIf}
yields our main theorem.

\begin{theorem}\label{ThIffi}
Let $H$ be a connected graph on at least 3 vertices and let $G$ be a
2-connected claw-free graph. Then $G$ satisfying $\varPhi(H,3)$
implies $G$ is hamiltonian, if and only if
$H=P_3,P_4,P_5,P_6,Z_1,B,N$ or $W$.
\end{theorem}

Note that the case of $H=N$ in Theorem \ref{ThIffi} shows that every
2-connected claw-free graph $G$ is hamiltonian if every vertex of
$G$ which is an end-vertex of an induced copy of $N$, has degree at
least $|V(G)|/3+1$. This gives an affirmative solution of the
conjecture of Broersma up to an additive constant.

\section{Some preliminaries}

Two famous conjectures in the field of hamiltonicity of graphs are
Thomassen's conjecture \cite{Thomassen} that every 4-connected line
graph is hamiltonian and Matthews and Sumner's conjecture
\cite{MatthewsSumner} that every 4-connected claw-free graph is
hamiltonian. Ryj\'a\v{c}ek proved these two conjectures are
equivalent. One major tool for the proof is his closure theory
\cite{Ryjacek}. Now we introduce Ryj\'a\v{c}ek's closure theory,
which we will use in our proof.

Let $G$ be a claw-free graph and $x$ a vertex of $G$. Following the
terminology of Ryj\'a\v{c}ek \cite{Ryjacek}, we call $x$ an
\emph{eligible} vertex if $N(x)$ induces a connected graph but is
not a clique in $G$. The \emph{completion} of $G$ at $x$, denoted by
$G'_x$, is the graph obtained from $G$ by adding all missing edges
$uv$ with $u,v\in N(x)$.

Note that if a vertex, say $v$, has a complete neighborhood in $G$,
i.e., $G[N(v)]$ is complete, then it also has a complete
neighborhood in $G'_x$; also note that if $P'$ is an induced path in
$G'_x$, then there is an induced path $P$ in $G$ with the same
end-vertices such that $V(P)\subset V(P')\cup\{x\}$.

Let $G$ be a claw-free graph. The \emph{closure} of $G$, denoted by
$cl(G)$, is the graph defined by a sequence of graphs
$G_1,G_2,\ldots,G_t$, and vertices $x_1,x_2,\ldots,x_{t-1}$ such
that
\begin{mathitem}
\item[(1)] $G_1=G$, $G_t=\cl(G)$;
\item[(2)] $x_i$ is an eligible vertex of $G_i$, $G_{i+1}=(G_i)'_{x_i}$,
$1\leq i\leq t-1$; and
\item[(3)] $G_t$ has no eligible vertices.
\end{mathitem}

By $c(G)$ we denote the length of a longest cycle of $G$.

\begin{theorem}[Ryj\'a\v{c}ek \cite{Ryjacek}]\label{ThRy}
Let $G$ be a claw-free graph. Then
\begin{mathitem}
\item[(1)] the closure $cl(G)$ is well-defined;
\item[(2)] there is a triangle-free graph $H$ such that $cl(G)$ is the
line graph of $H$; and
\item[(3)] $c(G) = c(\cl(G))$.
\end{mathitem}
\end{theorem}

Clearly every vertex has degree in $cl(G)$ not less than that in
$G$. Ryj\'a\v{c}ek proved that if $G$ is claw-free, then so is
$cl(G)$. A claw-free graph is said to be \emph{closed} if it has no
eligible vertices. The following properties of a closed claw-free
graph are obvious, and we omit the proofs.

\begin{lemma}\label{LeClosed}
Let $G$ be a closed claw-free graph. Then
\begin{mathitem}
\item[(1)] every vertex is contained in exactly one or two maximal
cliques;
\item[(2)] two distinct maximal cliques have at most one
common vertex;
\item[(3)] if two vertices are nonadjacent, then they have at most two common neighbors; and
\item[(4)] if a vertex has two neighbors in a maximal clique, then
it is contained in the clique.
\end{mathitem}
\end{lemma}

Now we introduce some new terminology which is useful for our proof.
Let $G$ be a claw-free graph and $K$ a maximal clique of $cl(G)$.
We call $G[K]$ a \emph{region} of $G$. For a vertex $v$ of $G$, we
call $v$ an \emph{interior vertex} if it is contained in only one
region, and a \emph{frontier vertex} if it is contained in two
distinct regions. For two vertices $u,v$ of $G$, we say that they
are \emph{associated} if they are in a common region, and
\emph{dissociated} otherwise. We use the notations $u\sim v$
($u\not \sim v$) to express the statement that $u$ and $v$ are
associated (dissociated). So two vertices are associated in $G$ if
and only if they are adjacent in $cl(G)$. Now we can reformulate
Lemma \ref{LeClosed} as follows.

\begin{lemma}\label{LeRegion1}
Let $G$ be a claw-free graph. Then
\begin{mathitem}
\item[(1)] every vertex is either an interior vertex of a region, or
a frontier vertex of two regions;
\item[(2)] every two regions are either disjoint or have only one
common vertex;
\item[(3)] every two dissociated vertices have at most two common neighbors; and
\item[(4)] if a vertex is associated with two vertices in a common
region, then it is also contained in the region.
\end{mathitem}
\end{lemma}

We can also get the following

\begin{lemma}\label{LeRegion}
Let $G$ be a claw-free graph. Then
\begin{mathitem}
\item[(1)] if $v$ is a frontier vertex of two regions $R,R'$,
then $N_R(v),N_{R'}(v)$ are cliques;
\item[(2)] if $R$ is a region of $G$, then $\cl(R)$ is complete;
\item[(3)] if $v$ is a frontier vertex and $R$ is a region
containing $v$, then $v$ has an interior neighbor in $R$ or $R$ is
complete and has no interior vertices; and
\item[(4)] if $u\sim v$, then there is an induced path
from $u$ to $v$ such that all internal vertices are interior
vertices in the region containing $u$ and $v$.
\end{mathitem}
\end{lemma}

\begin{proof}
(1) If there are two neighbors $x,x'$ of $v$ in $R$ such that
$xx'\notin E(G)$, then let $y$ be a neighbor of $v$ in $R'$. Note
that $y$ is nonadjacent to $x,x'$; otherwise it will be contained in
$R$. Now the subgraph induced by $\{v,x,x',y\}$ is a claw, a
contradiction. Thus $N_R(v)$, and similarly, $N_{R'}(v)$, is a
clique.

(2) Let $K=V(R)$. Let $G_1,G_2,\ldots,G_t$ be the sequence of
graphs, and $x_1,x_2,\ldots,x_{t-1}$ the sequence of vertices in the
definition of $cl(G)$. Note that for every $i\leq t-1$, $x_i$ has a
complete neighborhood in $G_{i+1}$, and then in $cl(G)$. This
implies that $x_i$ is an interior vertex. Thus if $x_i\notin K$,
then the completion of $G_i$ at $x_i$ does not change the structure
of $G_i[K]$. Let $x_{k_1},\ldots,x_{k_{t'-1}}$ be the subsequence of
$x_1,\ldots,x_{t-1}$ containing all vertices $x_{k_i}\in K$. Note
that $N_{G_{k_i}}(x_{k_i})\subset K$. Thus $x_{k_i}$ is an eligible
vertex of $G_{k_i}[K]$ and $(G_{k_i}[K])'_{x_{k_i}}=G_{k_i+1}[K]$.
Thus we have that $\cl(R)=\cl(G)[K]$ is the complete subgraph of
$cl(G)$ corresponding to $R$.

(3) If $R$ is complete in $G$, then either $v$ has an interior
neighbor in $R$ or $R$ has no interior vertices. Now we assume that
$R$ is not complete. By (2), $\cl(R)=\cl(G)[V(R)]$ is complete. This
implies that $R$ has at least one eligible vertex, and then, $R$ has
at least one interior vertex. If $v$ is nonadjacent to any interior
vertex in $R$, then the completion of an eligible vertex in $R$ does
not change the neighborhood of $v$. Thus $v$ will have no interior
neighbors in $R$ in the closure $\cl(R)$, a contradiction to that
$\cl(R)$ is a clique.

(4) Let $R$ be the region of $G$ containing $u$ and $v$. We use the
notation in the proof of (2). Note that for an induced path $P'$ in
$G_{k_{i+1}}[V(R)]$ connecting $u$ and $v$, there is also an induced
path $P$ in $G_{k_i}[V(R)]$ connecting $u$ and $v$ such that
$V(P)\subset V(P')\cup\{x_{k_i}\}$. This implies that there is an
induced path $P$ in $R$ connecting $u$ and $v$ such that
$V(P)\subset\{u,v\}\cup\{x_{k_i}: 1\leq i\leq t'-1\}$. Note that
every $x_{k_i}$ is an interior vertex of $R$. The proof is complete.
{\hfill$\Box$}
\end{proof}

In the case that $u\sim v$, we use $\varPi[uv]$ to denote an induced
path from $u$ to $v$ such that all internal vertices are interior
vertices in the region containing $u$ and $v$. For an induced path
$P=v_0v_1v_2\cdots v_k$ in $cl(G)$, we denote
$\varPi[P]=\varPi[v_0v_1]v_1\varPi[v_1v_2]v_2\cdots
v_{k-1}\varPi[v_{k-1}v_k]$ (note that $\varPi[P]$ is an induced path
of $G$).

Following \cite{Brousek}, we denote by $\mathcal{P}$ the class of
all graphs that are obtained by taking two disjoint triangles
$a_1a_2a_3a_1$, $b_1b_2b_3b_1$, and by joining every pair of
vertices $\{a_i,b_i\}$ by a path $P_{k_i}=a_ic_i^1c_i^2\cdots
c_i^{k_i-2}b_i$ for $k_i\geq 3$ or by a triangle $a_ib_ic_ia_i$. We
denote a graph from $\mathcal{P}$ by $P_{x_1,x_2,x_3}$, where
$x_i=k_i$ if $a_i$, $b_i$ are joined by a path $P_{k_i}$, and
$x_i=T$ if $a_i$, $b_i$ are joined by a triangle.

\begin{theorem}[Brousek \cite{Brousek}]\label{ThBr}
Every non-hamiltonian 2-connected claw-free graph contains an
induced subgraph in $\mathcal{P}$.
\end{theorem}

We mention the following result deduced from Brousek et al.
\cite{BrousekRyjacekFavaron} to complete this section.

\begin{theorem}[Brousek et al.
\cite{BrousekRyjacekFavaron}]\label{ThBrRyFa} Let $G$ be a claw-free
graph. If $G$ is $N$-free, then $cl(G)$ is also $N$-free.
\end{theorem}

\section{Proof of Theorem \ref{ThIf}}

Assume that $G$ is not hamiltonian. By Theorems \ref{ThRy} and
\ref{ThBr}, $cl(G)$ contains an induced subgraph
$P_{x_1,x_2,x_3}\in\mathcal {P}$. We use the notation $a_i,b_i,c_i$
and $c_i^j$ defined in Section 2. If $x_i=k_i$, then let $P^i$ be
the path $a_ic_i^1c_i^2\cdots c_i^{k_i-2}b_i$; if $x_i=T$, then let
$P^i=a_ib_i$. Let $A$ be the region of $G$ containing the vertices
$a_1,a_2,a_3$, $B$ be the region of $G$ containing the vertices
$b_1,b_2,b_3$. Note that $A$ and $B$ are possibly not disjoint. If
they are not disjoint, then let $c$ be the only common vertex of $A$ and
$B$. Clearly, $a_i,b_i$ and $c$ (if exists) are all frontier
vertices. If $x_i=T$, then let $a'_i$ be the successor of $a_i$ in
$\varPi[a_ic_i]$ and $b'_i$ be the successor of $b_i$ in
$\varPi[b_ic_i]$; if $x_i=k_i$, then let $a'_i$ be the successor of
$a_i$ in $\varPi[a_ic_i^1]$ and $b'_i$ be the successor of $b_i$ in
$\varPi[b_ic_i^{k_i-2}]$.

In this section, we say that a vertex is \emph{hefty} if it has
degree at least $n/3+1$.

\begin{claim}\label{ClHefty}
Let $v_1,v_2,v_3$ be three pairwise nonadjacent vertices of $G$.
\begin{mathitem}
\item[(1)] If $v_1\not\sim v_2$, $v_1\not\sim v_3$ and $v_2,v_3$
have at most one common neighbor, then one of $v_1,v_2,v_3$ is not
hefty.
\item[(2)] If $v_1\not\sim v_2$, $v_1\not\sim v_3$ and $v_2\not\sim v_3$, then
one of $v_1,v_2,v_3$ is not hefty.
\end{mathitem}
\end{claim}

\begin{proof}
(1) By Lemma \ref{LeRegion1} (3), $|N(v_1)\cap N(v_2)|\leq 2$ and
$|N(v_1)\cap N(v_3)|\leq 2$. Note that $|N(v_2)\cap N(v_3)|\leq 1$.
If all these three vertices are hefty, i.e., $d(v_i)\geq n/3+1$ for
$i=1,2,3$, then
\begin{align*}
n\geq 3+\sum_{1\leq i\leq 3}d(v_i)-\sum_{1\leq i<j\leq 3}|N(v_i)\cap
N(v_j)|\geq 3+3\left(\frac{n}{3}+1\right)-5=n+1,
\end{align*}
a contradiction.

(2) By (1) and Lemma \ref{LeRegion1} (3), each of
$\{v_1,v_2\},\{v_1,v_3\},\{v_2,v_3\}$ has exactly two common
neighbors. Let $u_{ij}$ and $u'_{ij}$ be the two common neighbors of
$v_i$ and $v_j$. By Lemma \ref{LeRegion1} (4), $u_{ij}\not\sim
u'_{ij}$. This implies that all the three vertices $v_1,v_2,v_3$ are
frontier vertices. Moreover, by applying a similar argument as in
(1), we have
$$n\geq 3+d(v_1)+d(v_2)+d(v_3)-6\geq
3\cdot\left(\frac{n}{3}+1\right)-3=n.$$ This implies that every
vertex of $G$ is adjacent to at least one vertex in
$\{v_1,v_2,v_3\}$. Thus $G$ consists of the six regions containing
$v_1,v_2$ and $v_3$, and all the six regions are cliques by Lemma
\ref{LeRegion} (1).

Since $u_{12}\not\sim u'_{12}$ and $u_{13}\not\sim u'_{13}$ and all the
four vertices are adjacent to $v_1$, we have either $u_{12}\sim
u_{13}$ or $u_{12}\sim u'_{13}$. We assume without loss of
generality that $u_{12}\sim u_{13}$, which implies that
$u_{12}u_{13}\in E(G)$. Now we can begin with the cycle
$C_0=v_1u'_{12}v_2u_{12}u_{13}v_3u'_{13}v_1$, and add other vertices,
one by one, to the cycle at the place between two associated
vertices, and finally obtain a Hamilton cycle of $G$, a contradiction.
{\hfill$\Box$}
\end{proof}

\subsection*{The case $H=P_6$}

Let $P=a'_1a_1\varPi[a_1a_2]a_2\varPi[P^2]b_2\varPi[b_2b_3]b_3b'_3$.
Note that $P$ is an induced copy of $P_l$ with $l\geq 6$. This
implies that $a'_1$, and similarly, $a'_2,a'_3$, are hefty. Note
that $a'_1,a'_2$ and $a'_3$ are pairwise dissociated in $G$, a
contradiction to Claim \ref{ClHefty}.

\subsection*{The case $H=N$}

\begin{claim}\label{ClNAB}
There are at least two hefty vertices in $A$ (and similarly, in
$B$).
\end{claim}

\begin{proof}
Let $G'=G[V(A)\cup\{a'_1,a'_2,a'_3\}]$. From Lemma \ref{LeRegion}
(2), we can see that $\cl(G')=\cl(G)[V(G')]$. Note that the subgraph
of $\cl(G)[V(G')]$ induced by $\{a_1,a'_1,a_2,a'_2,a_3,a'_3\}$ is an
$N$. By Theorem \ref{ThBrRyFa}, $G'$ contains an induced $N$. This
implies that $V(G')$ contains at least three pairwise nonadjacent
hefty vertices. If two of them are not in $A$, then we assume
without loss of generality that $a'_1,a'_2$ are hefty. Note that the
third hefty vertex is in $(V(A)\cup\{a'_3\})\backslash\{a_1,a_2\}$.
This implies that the three hefty vertices are pairwise dissociated,
a contradiction to Claim \ref{ClHefty}. {\hfill$\Box$}
\end{proof}

Let $b,b'$ be two hefty vertices in $B$. Set $$N_i=\{v\in V(A):
d_A(a_1,v)=i\} \mbox{ and } j=\max\{i: N_i\neq\emptyset\}.$$ Note
that $N_0=\{a_1\}$ and $N_1=N_{A}(a_1)$. In addition, we define that
$N_{-1}=\{a'_1\}$. Note that for any vertex $v\in N_i$, with $1\leq
i\leq j$, $v$ has a neighbor in $N_{i-1}$. Also note that if $v$ has
a neighbor in $N_{i+1}$, $1\leq i\leq j-1$, then by Lemma
\ref{LeRegion} (1), $v$ is an interior vertex, especially, $v$ is
not $a_2,a_3$ and $c$.

\begin{claim}\label{ClNClique}
  $N_i$ is a clique for all $1\leq i\leq j$.
\end{claim}

\begin{proof}
We use induction on $i$. By Lemma \ref{LeRegion} (1), $N_1$ is a
clique. Now we assume that $2\leq i\leq j$. Note that
$N_{i-1},N_{i-2}$ and $N_{i-3}$ are nonempty.

Assume that there are two vertices $y,y'$ in $N_i$ with $yy'\notin
E(G)$. If $y$ and $y'$ have a common neighbor in $N_{i-1}$, then let
$x$ be a common neighbor of $y$ and $y'$ in $N_{i-1}$, and $w$ be a
neighbor of $x$ in $N_{i-2}$. Then the subgraph induced by
$\{x,w,y,y'\}$ is a claw, a contradiction. This implies that $y$ and
$y'$ have no common neighbors in $N_{i-1}$. Now let $x$ be a
neighbor of $y$ in $N_{i-1}$ and $x'$ be a neighbor of $y'$ in
$N_{i-1}$. Note that $xy',x'y\notin E(G)$. Let $w$ be a neighbor of
$x$ in $N_{i-2}$ and let $v$ be a neighbor of $w$ in $N_{i-3}$. By
the induction hypothesis, $xx'\in E(G)$. If $wx'\notin E(G)$, then
the subgraph induced by $\{x,w,x',y\}$ is a claw, a contradiction.
This implies that $wx'\in E(G)$. Now the subgraph induced by
$\{w,v,x,y,x',y'\}$ is an $N$. Thus the three vertices $v,y$ and
$y'$ are all hefty.

By Lemma \ref{LeRegion1} (4), $v\not\sim b$ or $v\not\sim b'$. We assume
without loss of generality that $v\not\sim b$. Similarly $b\not\sim y$ or
$b\not\sim y'$, we assume without loss of generality that $b\not\sim y$.
Note that $b,v,y$ are all hefty, $b\not\sim v$, $b\not\sim y$ and $v,y$
have no common neighbors. We get a contradiction. {\hfill$\Box$}
\end{proof}

If both $a_2$ and $a_3$ are in $N_j$, then let $w$ be a neighbor of
$a_2$ in $N_{j-1}$, $v$ be a neighbor of $w$ in $N_{j-2}$. By Claim
3 and Lemma \ref{LeRegion} (1), $a_2a_3,wa_3\in E(G)$. Thus the
subgraph induced by $\{w,v,a_2,a'_2,a_3,a'_3\}$ is an $N$. Thus
$v,a'_2$ and $a'_3$ are three hefty vertices. Note that $v,a'_2$ and
$a'_3$ are pairwise dissociated, a contradiction. So we assume
without loss of generality that $a_2\notin N_j$.

Let $a_2\in N_i$, where $1\leq i\leq j-1$. Let $y$ be a vertex in
$N_{i+1}$. Recall that $a_2$ has no neighbors in $N_{i+1}$. Let $x$
be a neighbor of $y$ in $N_i$, $w$ be a neighbor of $a_2$ in
$N_{i-1}$ and $v$ be a neighbor of $w$ in $N_{i-2}$. By Claim
\ref{ClNClique} and Lemma \ref{LeRegion} (1), $a_2x,wx\in E(G)$, and
the subgraph induced by $\{w,v,x,y,a_2,a'_2\}$ is an $N$. Thus $v,y$
and $a'_2$ are three hefty vertices. Note that $a'_2\not\sim v$,
$a'_2\not\sim y$, and $v,y$ have no common neighbors, a contradiction.

\subsection*{The case $H=W$}

\begin{claim}\label{ClWNoedge}
For $i,j$, $1\leq i<j\leq 3$, one of the edges in
$\{a_ia_j,b_ib_j,a_ib_i,a_jb_j\}$ is not in $E(G)$.
\end{claim}

\begin{proof}
We assume that $a_ia_j,b_ib_j,a_ib_i,a_jb_j\in E(G)$. By Lemma
\ref{LeRegion} (1), $a'_ib_i,a'_jb_j\in E(G)$. Let $a$ be the
successor of $a_j$ in the path $\varPi[a_ja_k]$, where $k\neq i,j$.
Then the subgraph induced by $\{a'_j,a_j,a,b_j,b_i,a'_i\}$ is a $W$.
Thus $a,a'_i$, and similarly $a'_j$, are hefty. Note that $a,a'_i$
and $a'_j$ are pairwise dissociated, a contradiction. {\hfill$\Box$}
\end{proof}

As in the case of $N$, we set $$N_i=\{v\in V(A): d_A(a_1,v)=i\}
\mbox{ and } j=\max\{i: N_i\neq\emptyset\}.$$ Note that
$N_0=\{a_1\}$, $N_1=N_{A}(a_1)$ and we define additionally
$N_{-1}=\{a'_1\}$.

\begin{claim}\label{ClWAB}
There is a hefty vertex in $A\backslash\{a_1,a_2,a_3,c\}$ (and
similarly, in $B\backslash\{b_1,b_2,b_3,c\}$).
\end{claim}

\begin{proof}
We assume on the contrary that there are no hefty vertices in
$A\backslash\{a_1,a_2,a_3,c\}$.

\begin{subclaim}\label{ClWABClique}
  $N_i$ is a clique for all $1\leq i\leq j$.
\end{subclaim}

\begin{proof}
We use induction on $i$. By Lemma \ref{LeRegion} (1), $N_1$ is a
clique. Now we assume that $2\leq i\leq j$. Note that
$N_{i-1},N_{i-2}$ and $N_{i-3}$ are nonempty.

Assume that there are two vertices $y,y'$ in $N_i$ with $yy'\notin
E(G)$. Note that $y$ and $y'$ have no common neighbors in $N_{i-1}$.
Let $x$ be a neighbor of $y$ in $N_{i-1}$, $x'$ be a neighbor of
$y'$ in $N_{i-1}$, $w$ be a neighbor of $x$ in $N_{i-2}$ and $v$ be
a neighbor of $w$ in $N_{i-3}$. By the induction hypothesis, $xx'\in
E(G)$. Note that $wx'\in E(G)$; otherwise the subgraph induced by
$\{x,w,x',y\}$ is a claw.

If $y=a_2$, then the subgraph induced by $\{x',w,v,x,a_2,a'_2\}$ and
the subgraph induced by $\{w,x',y',x,a_2,a'_2\}$ are $W$'s. Thus
$v,y'$ and $a'_2$ are three hefty vertices. Note that $a'_2\not\sim v$
$a'_2\not\sim y'$, and $v,y'$ have no common neighbors, a
contradiction. So we assume that $y\neq a_2$, and similarly, $y\neq
a_3$, $y'\neq a_2$, $y'\neq a_3$. This implies that either $y$ or
$y'$ is in $A\backslash\{a_1,a_2,a_3,c\}$.

We assume without loss of generality that $y\in
A\backslash\{a_1,a_2,a_3,c\}$. Let $P'$ be a shortest path from $w$
to $a_1$ (note that $P'$ consists of the vertex $a_1$ if $w=a_1$).
Let $w,v$ and $u$ be the first three vertices in the path
$P=P'a_1\varPi[P^1]b_1\varPi[b_1b_2]$. Then the subgraph induced by
$\{x',x,y,w,v,u\}$ is a $W$. Thus $y$ is a hefty vertex, a
contradiction. {\hfill$\Box$}
\end{proof}

If both $a_2$ and $a_3$ are in $N_j$, then let $w$ be a neighbor of
$a_2$ in $N_{j-1}$, $v$ be a neighbor of $w$ in $N_{j-2}$. By Claim
\ref{ClWABClique} and Lemma \ref{LeRegion} (1), $a_2a_3,wa_3\in
E(G)$. Let $a_2,y$ and $z$ be the first three vertices in the path
$P=\varPi[P^2]b_2\varPi[b_2b_3]$. By Claim \ref{ClWNoedge},
$a_3z\notin E(G)$. Then the subgraph induced by
$\{a_3,w,v,a_2,y,z\}$ is a $W$. Let $a_3,y',z'$ be the first three
vertices in the path $P=\varPi[P^2]b_2\varPi[b_2b_1]$. By Claim
\ref{ClWNoedge}, $wz'\notin E(G)$. Then the subgraph induced by
$\{w,a_2,a'_2,a_3,y',z'\}$ is a $W$. Thus $v,a'_2$, and similarly,
$a'_3$, are hefty. Note that $v,a'_2$ and $a'_3$ are pairwise
dissociated, a contradiction. So we assume without loss of
generality that $a_2\notin N_j$.

Let $a_2\in N_i$, where $1\leq i\leq j-1$. Let $y$ be a vertex in
$N_{i+1}$. Recall that $a_2$ has no neighbors in $N_{i+1}$. Let $x$
be a neighbor of $y$ in $N_i$, $w$ be a neighbor of $a_2$ in
$N_{i-1}$ and $v$ be a neighbor of $w$ in $N_{i-2}$. Note that
$a_2x,wx\in E(G)$.

If $y=a_3$, then let $z=a'_3$; and if $y=c$, then let $z$ be the
successor of $c$ in $\varPi[cb_3]$. Then the subgraph induced by
$\{a_2,w,v,x,y,z\}$ and the subgraph induced by
$\{w,a_2,a'_2,x,y,z\}$ are $W$'s. Thus $v,a'_2$ and $z$ are hefty.
Note that $v,a'_2$ and $z$ are pairwise dissociated, a
contradiction. Now we assume that $y\neq c,a_3$. Let $a_2,y',z'$ be
the first three vertices in the path
$P=\varPi[P^2]b_2\varPi[b_2b_3]$. Then the subgraph induced by
$\{w,x,y,a_2,y',z'\}$ is a $W$. This implies that $y$ is hefty, a
contradiction. {\hfill$\Box$}
\end{proof}

Now let $a$ and $b$ be two hefty vertices in
$A\backslash\{a_1,a_2,a_3,c\}$ and $B\backslash\{b_1,b_2,b_3,c\}$,
respectively. Since $a,b$ and $a'_i$ are pairwise dissociated,
$a'_i$ is not hefty.

By Lemma \ref{LeRegion} (3), $a_1$ has an interior neighbor in $A$
or $a_1a\in E(G)$. In any case, $a_1$ has a neighbor in
$A\backslash\{a_2,a_3,c\}$. If $a_1a_2\in E(G)$, then let $v$ be a
neighbor of $a_1$ in $A\backslash\{a_2,a_3,c\}$. By Lemma
\ref{LeRegion}, $a_2v\in E(G)$. Let $a_2,x$ and $y$ be the first
three vertices in the path $P=\varPi[P^2]b_2\varPi[b_2b_3]$. Then
the subgraph induced by $\{v,a_1,a'_1,a_2,x,y\}$ is a $W$. Thus
$a'_1$ is hefty, a contradiction. This implies that $a_1a_2$, and
similarly, $a_1a_3,a_2a_3$, is not in $E(G)$.

\begin{claim}\label{ClWClique}
  $N_i$ is a clique for all $1\leq i\leq j$.
\end{claim}

\begin{proof}
We use induction on $i$. By Lemma \ref{LeRegion} (1), $N_1$ is a
clique.

We first consider the case $i=2$. Recall that $a_1a_2\notin E(G)$,
which implies that $a_2\notin N_1$. If $a_2\in N_2$, then let
$z=a'_2,y=a_2$; and if $a_2\notin N_2$, then ($j\geq 3$ and) let $z$
be a vertex in $N_3$, and $y$ be a neighbor of $z$ in $N_2$.

We claim that $y$ is adjacent to every vertex in
$N_2\backslash\{y\}$. Assume that $yy'\notin E(G)$ for $y'\in
N_2\backslash\{y\}$. Then $y$ and $y'$ have no common neighbors in
$N_1$. Let $x$ be a neighbor of $y$ in $N_1$ and $x'$ be a neighbor
of $y'$ in $N_1$. Then $xy',x'y\notin E(G)$. Since $xx'\in E(G)$,
the subgraph induced by $\{x',a_1,a'_1,x,y,z\}$ is a $W$, and this
implies that $a'_1$ is hefty, a contradiction. Thus, as we claimed,
$y$ is adjacent to every vertex in $N_2\backslash\{y\}$. Now let
$y',y''$ be two vertices in $N_2\backslash\{y\}$. We claim that
$y'y''\in E(G)$. If $y'z\in E(G)$, then ($z\neq a'_2$ and) similarly
as the case of $y$, we can see that $y'$ is adjacent to every vertex
in $N_2\backslash\{y'\}$, including $y''$. So we assume that $y'z$,
and similarly, $y''z$, is not in $E(G)$. Then the subgraph induced
by $\{y,y',y'',z\}$ is a claw, a contradiction. Thus, as we claimed,
$N_2$ is a clique.

Now we assume that $3\leq i\leq j$. Note that
$N_{i-1},N_{i-2},N_{i-3}$ and $N_{i-4}$ are nonempty.

Assume that there are two vertices $z$ and $z'$ in $N_i$ with
$zz'\notin E(G)$. Note that $z$ and $z'$ have no common neighbors in
$N_{i-1}$. Let $y$ be a neighbor of $z$ in $N_{i-1}$ and $y'$ be a
neighbor of $z'$ in $N_{i-1}$. Then $yz',y'z\notin E(G)$. Let $x$ be
a neighbor of $y$ in $N_{i-2}$, $w$ be a neighbor of $x$ in
$N_{i-3}$ and $v$ be a neighbor of $w$ in $N_{i-4}$. Then
$yy',xy'\in E(G)$. Now the subgraph induced by $\{y',y,z,x,w,v\}$ is
a $W$. Thus $v$ and $z$ are hefty. Note that $b\not\sim v$, $b\not\sim z$,
and $v,z$ have no common neighbors, a contradiction. {\hfill$\Box$}
\end{proof}

Recall that $a_2a_3\notin E(G)$, which implies that either $a_2$ or
$a_3\notin N_j$. Also recall that $a_2,a_3\notin N_1$. We assume
without loss of generality that $a_2\in N_i$, where $2\leq i\leq
j-1$. Let $z$ be a vertex in $N_{i+1}$, $y$ be a neighbor of $z$ in
$N_i$, $x$ be a neighbor of $a_2$ in $N_{i-1}$, $w$ be a neighbor of
$x$ in $N_{i-2}$ and $v$ be a neighbor of $w$ in $N_{i-3}$. By Claim
\ref{ClWClique} and Lemma \ref{LeRegion} (1), $a_2y,xy\in E(G)$.
Then the subgraph induced by $\{y,a_2,a'_2,x,w,v\}$ is a $W$. This
implies that $a'_2$ is hefty, a contradiction.

The proof is complete. {\hfill$\Box$}

\end{document}